\documentclass[12pt]{amsart}

\usepackage{amssymb}

\usepackage{enumitem}

\usepackage{graphicx}
\usepackage{amsthm,amsmath,amssymb}
\usepackage{mathrsfs}
\usepackage{color}

\makeatletter
\@namedef{subjclassname@2020}{%
  \textup{2020} Mathematics Subject Classification}
\makeatother

\usepackage[T1]{fontenc}
\newtheorem{theorem}{Theorem}[section]

\newtheorem{lemma}[theorem]{Lemma}

\theoremstyle{definition}

\newtheorem*{xrem}{Remark}

\numberwithin{equation}{section}

\begin{document}

%%%%% To ease editing, for IMPAN journals add:

\baselineskip=17pt

%%%%%%%%%%%%%%%%

\title[Sample paper]{Least zero of a cubic form}

\author[Y. X. Xiao]{Yixiu Xiao}
%%%%%\address{Institute of Mathematics\\ Polish Academy of Sciences\\
%%%%%     \'Sniadeckich 8\\
%%%%%     00-656 Warszawa, Poland}
%%%%
\address{School of Mathematical Sciences\\ Shanghai Jiao Tong University\\
800 Dongchuan RD\\
200240 Shanghai, China}
\email{asaka1312@sjtu.edu.cn}

\author[H. Z. Li]{Hongze Li}
\address{School of Mathematical Sciences\\ Shanghai Jiao Tong University\\
800 Dongchuan RD\\
200240 Shanghai, China}
\email{lihz@sjtu.edu.cn}

\date{}

\begin{abstract}
An explicit upper bound is established for the least non-trivial integer zero of an arbitrary cubic form $C \in \mathbb{Z}[X_1,...,X_n],$ provided that $n \geq 14.$
\end{abstract}

\subjclass[2020]{11D72; 11D25; 11P55}

\keywords{cubic forms, least non-trivial zero, circle method}

\maketitle

\section{Introduction}
Let $n \geq 3$ and $F \in \mathbb{Z}[X_1,...,X_n]$ be an indefinite cubic form, with coefficients of maximum modulus $\left\| F\right\|$ and greatest common divisor $1$. It is very natural to determine whether or not the equation $F = 0$ is soluble in integers. When there
is a non-zero solution $\boldsymbol{x} = (x_1,...,x_n) \in \mathbb{Z}^n$ to the equation $F = 0,$ let $\Lambda_n (F)$ denote the smallest positive integer $A$ such that there exists a non-zero solution
$\boldsymbol{x} = (x_1,...,x_n) \in \mathbb{Z}^n$ to the equation $F = 0$ with $\max_{1 \leq i \leq n}|x_i| \leq A.$ For the
history of this problem, we refer the readers to \cite{browning2009least}.

The best known result is due to Browning, Dietmann and Elliott \cite{browning2009least}, in which they obtain a good upper bound for $\Lambda_n (C)$ when $n \geq 17.$ They show that for a cubic form $C \in \mathbb{Z}[X_1,...,X_n]$ with $n \geq 17,$ defining a hypersurface with at most isolated ordinary singularities, then for any $\varepsilon >0,$ one has
\begin{equation*}
        \Lambda_n(C) \leq c_{n, \varepsilon} \left\| C\right\|^{e_1(n)+\varepsilon},
\end{equation*}
where $c_{n, \varepsilon} > 0$ is some constant, and
\begin{equation*}
        e_1(n)=
        \begin{cases}
        \frac{22n^3+107n^2-597n-432}{(n-2)(n-9)(n-16)}, \quad &\text{if }17 \leq n \leq 20,\\
        \frac{n^4+125n^3+1518n^2-7236n-4320}{32(n-2)(n-9)}, \quad &\text{if }n > 20.
    \end{cases}
    \end{equation*}
Especially one has
    \begin{equation*}
        \Lambda_{17}(C) \leq c \left\| C\right\|^{1071}.
    \end{equation*}
And for large $n,$ one has $e_1(n) \sim \frac{n^2}{32}. $

For arbitrary cubic form $C \in \mathbb{Z}[X_1,...,X_n]$ with $n \geq 17,$ they also prove that
\begin{equation*}
        \Lambda_n(C) \leq c \left\| C\right\|^{360000},
\end{equation*}
for some absolute constant $c>0.$

In this paper, we consider the cubic form in $n \geq 14$ variables, and give the following result.

%\begin{xrem}
%Noting that $e_2(16) \approx 304.23,$ and when $n \geq 16$, $e_2(n)=158+\frac{1074}{n-2}+\frac{6813}{(n-2)(n-9)}$,  so
% $e_2(n)$ is a monotonically decreasing function of $n$ with $n \geq 14$. 
%\end{xrem}

\begin{theorem}

    \label{theorem_normal}
    Let $C \in \mathbb{Z}[X_1,...,X_n]$ be a cubic form, with $n \geq 14.$ Then there exist some constant $c_{n, \varepsilon} >0$ such that
    \begin{equation*}
        \Lambda_n(C) \leq c_{n, \varepsilon} \left\| C\right\|^{e_2(n)+\varepsilon},
    \end{equation*}
    where the exponents $e_2(n)$ are given by the Table \ref{tab1}.
    
    \begin{table}[htbp]%调节图片位置，h：浮动；t：顶部；b:底部；p：当前位置
        \begin{center}
        \caption{The values of $e_2(n)$ for $n \geq 14$}\label{tab1}\vspace{-2mm}
        \end{center}
  
    	\centering
    	%\caption{Relative running time of the considered filters}
    	%\label{tab:1}  
    	\begin{tabular}{l|c|c|c|c|c|}
    		\hline\hline\noalign{\smallskip}	
    		$n$ & $14$ & $15$ & $16$ & $17$ & $\geq 18$\\
    		\noalign{\smallskip}\hline\noalign{\smallskip}
    		$e_2(n)$ & $138500$ & $87844$ & $74851$ & $71400$ & $70932$\\ 
    		\noalign{\smallskip}\hline
    	\end{tabular}
    \end{table}
    
\end{theorem}

After the first version of this paper, Professor T.D.Browning mention the work \cite{bernert2023} of C.Bernert to us. Since the lower bound for singular series in Lemma 11 of \cite{browning2009least} is incorrect, we cited the first part of Lemma 16 of Bernert as  Lemma 4.1. After the modification, our final result of $n =14$ seems to be worse than that of C.Bernert, but there may be an error in his article. Specifically, in his Theorem 4, his selection of $P_0 $ did not satisfy the condition $(\mathfrak {S}_4)$ in p30, this means that the bound is incorrect in his Theorem 4. And more, we give the upper bound for all
$n\geq 14$.

For the notations in this paper, we shall use the symbols $c$ or $c_{n, \varepsilon}$ to denote various positive real constants, not necessarily the same at each occurrence. All of the implied constants in our work will be allowed to depend on $\varepsilon$ and $n,$ with any further dependence being made completely explicit. We will allow $\varepsilon$ to take different values at different parts of the argument, but we shall always assume that it is very small. Throughout the remainder of the paper, we denote
\begin{equation*}
    M=\left\| C\right\|,
\end{equation*}
as the height of the cubic form $C,$ and $|\boldsymbol{x}|$ as the norm $\max_{1 \leq i \leq n}|x_i|$ of any vector $\boldsymbol{x} \in \mathbb{R}^n.$ We adopt the notation
\begin{equation*}
    \left\| \alpha \right\| = \min_{n \in \mathbb{Z}}|\alpha-n|
\end{equation*}
to denote the distance of a real number $\alpha$ from integers. $\lfloor x \rfloor$ and $\lceil x \rceil$ denote the greatest integer $\leq x$ and the least integer $\geq x,$ respectively.
Finally, we will assume that the parameter $M$ is a sufficiently large integer, and $\mathcal{L}=\log M.$

\section{The Outline of the Proof of the Theorems}

For the proof of the theorem, we follow the approach in \cite{browning2009least}. 
%For the processing of the major arcs and the truncated singular series, we use the methods of \cite{browning2009least} and \cite{bernert2023}. 
And we use the method of \cite{heath2007cubic} to get explicit upper bound for the minor arcs.

We write our cubic form as
\begin{equation}
    \label{C}
    C(x_1,...,x_n)=\sum_{i,j,k}c_{ijk}x_ix_jx_k,
\end{equation}
in which the coefficients $c_{ijk} \in \mathbb{Z}$ are symmetric in the indices $i,j,k$. Without loss of generality, we may assume that 
\begin{equation}
    \label{c111}
    c_{111} \gg M,
\end{equation}
for an absolute implied constant.

We define an $n \times n$ matrix $M(\boldsymbol{x})$ as follows
\begin{equation}
    \label{M(x)}
    M(\boldsymbol{x})_{jk}=\sum_{1 \leq i \leq n}c_{ijk}x_i.
\end{equation}
And let $r(\boldsymbol{x})=\text{rank}(M(\boldsymbol{x})),$ if for any $\varepsilon>0,$ the following
\begin{equation}
    \label{good}
    \#\{\boldsymbol{x} \in \mathbb{Z}^{n}:|\boldsymbol{x}| \leq H, r(\boldsymbol{x})=r\} \leq H^{r+\varepsilon}
\end{equation}
holds for each $0 \leq r \leq n$ and any $H$ in $1 \leq H \leq M^{\vartheta},$ we will call the cubic form $C$ $\vartheta$-good for $1 \leq \vartheta \leq \infty.$

\begin{lemma}
\label{lemma_when_not_good}
Let $n \geq 3$ and $\vartheta \geq 1$. Then either $C$ is $\vartheta$-good or
\begin{equation}
    \label{when_not_good}
    \Lambda_{n}(C) \ll M^{\frac{n^2-1}{2}+\frac{\vartheta n(n-1)}{2}}.
\end{equation}
\end{lemma}

This is the Proposition 1 of \cite{browning2009least}.

For a $\vartheta$-good cubic form $C$, we will give an upper bound for $\Lambda_{n}(C)$ by circle method. For a given $\boldsymbol{z} \in \mathbb{R}^n$ and $0 < \rho < 1$, let
\begin{equation*}
    \mathscr{B}=\mathscr{B}(\boldsymbol{z},\rho) =\prod_{i=1}^{n}[z_i-\rho,z_i+\rho].
\end{equation*}

From now on we assume that
\begin{equation}
    \label{rho,z}
    \rho \asymp M^{-2-\frac{5}{n-2}},\quad |\boldsymbol {z}| \ll M^{\frac{1}{n-2}},
\end{equation}
where the choices of $\rho, \boldsymbol {z}$ come from Lemma \ref{choice}.

We define the generating function
\begin{equation*}
    S(\alpha) = \sum_{\boldsymbol{x}\in \mathbb{Z}^{n}\cap P\mathscr{B}} e(\alpha C(\boldsymbol{x})),
\end{equation*}
where $e(\alpha) := \exp(2\pi i\alpha)$ as usual. We will always assume that $ P \geq 1,$ so that the above sum is non-trivial.

Our starting point is
 \begin{equation}
    \label{N(P)dingyi}
     N(P)=N_{\mathscr{B}}(P,C)=\sum_{\substack{\boldsymbol {x}\in \mathbb{Z}^{n}\cap P\mathscr{B}\\[3pt]C(\boldsymbol {x})=0}} 1=\int_0^1 S(\alpha) d\alpha. 
 \end{equation}
%Let \textcolor{red}{for infty good, while not enough for normal good}
%\begin{equation} 
 %   \label{P_0_definition}
 %   P_{0,\infty}=M^{\frac{37n^4-213n^3+328n^2-292n+96}{(n-2)(n^3-17n^2+54n-32)}+\varepsilon}.
%\end{equation}
%The selection \eqref{P_0_definition} for $P_{0,\infty}$ comes from \eqref{intermediate_range_left}.
Let %{P_0_another_requirment} is always stronger than {intermediate_range_left}
\begin{equation} 
    \label{P_0_vartheta_definition}
    P_{0} = M^{\frac{3675n\sqrt{3n(3n-35)}}{(\sqrt{3n(3n-35)}-18n+210)(3n-6\sqrt{3n(3n-35)})}+\varepsilon}.
\end{equation}
For simplicity, we sometimes write $P_0=M^{eoP_0}$ in the following argument, where
\begin{equation}
    \label{eop_definition}
    eoP_0=\frac{3675n\sqrt{3n(3n-35)}}{(\sqrt{3n(3n-35)}-18n+210)(3n-6\sqrt{3n(3n-35)})}+\varepsilon.
\end{equation}

The selection \eqref{P_0_vartheta_definition} for $P_0$ comes from \eqref{P_0_another_requirment} and \eqref{intermediate_range_left}. 

We define the major arcs
\begin{equation*}
    \mathfrak{M}=\bigcup\limits_{q\leq P_0}\bigcup_{\substack{1\leq a\leq q\\[3pt](a,q)=1}}\mathfrak{M}(a,q),
\end{equation*}
where
\begin{equation*}
    \mathfrak{M}(a,q)=\left[\frac{a}{q}-T,\frac{a}{q}+T\right],
\end{equation*}
and the minor arcs $\mathfrak{m}=[0,1] \backslash\mathfrak{M}$, defined module 1. The union of major arcs will be disjoint provided that
\begin{equation}
\label{disjoint_preliminary}
    2TP_0^2 < 1.
\end{equation}

Let 
\begin{equation}  %done3
    \label{u_normal_definition}
    T =
    %u_14_15=(11*n^4 + 37*n^3 - 136*n^2 + 212*n - 64)/(n^4 - 19*n^3 + 88*n^2 - 140*n + 64) 
    \begin{cases}
        P^{-3+\varepsilon}M^{320.239}, \quad &\text{if } n =14,\\
        P^{-3+\varepsilon}M^{\frac{4(3n^2-2n+1)}{n-2}}, \quad &\text{if } n \geq 15.\\
    \end{cases}
\end{equation}
\eqref{u_normal_definition} comes from \eqref{u_need_from_major_integral_error_term} and \eqref{u_from_small_denominator_large_distance}.

Let
\begin{equation*}
    P=M^{e(n)}.
\end{equation*}
Substituting in our choice \eqref{u_normal_definition} of $T$ and \eqref{P_0_vartheta_definition} of $P_{0},$ \eqref{disjoint_preliminary} holds for $n \geq 14$ provided 
\begin{equation} %done3
    \label{major_arcs_disjoint_vartheta}
    e(n) \geq
    \begin{cases}
        495.446, \quad &\text{if } n =14,\\
       \frac{4(3n^2-2n+1)}{3(n-2)}+\frac{2eoP_0}{3}, \quad &\text{if } n \geq 15,\\
    \end{cases}
\end{equation}

For any $R \geq 1$, we also define the truncated singular series, given by
\begin{equation}
    \label{singular_series}
    \mathfrak{S}(R)=\sum_{q\leq R}\sum_{\substack{1\leq a\leq q\\[3pt](a,q)=1}} \frac{S(a,q)}{q^{n}}=\sum_{q\leq R}\sum_{\substack{1\leq a\leq q\\[3pt](a,q)=1}} q^{-n}\sum_{\boldsymbol {r}(\text{mod } q)} e_q(aC(\boldsymbol {r})),
\end{equation}
 where $e_q(x)=\exp(\frac{2\pi i x}{q}).$

For the contribution from the major arcs, we have the following Lemma.
\begin{lemma}
    \label{contribution_major}
Let $\varepsilon > 0$ and assume that $n \geq 14,$
\begin{equation}
\label{The child I want to abandon}
    e(n) \geq
    \begin{cases}
    904.479, \quad &\text{if } n =14,\\
    eoP_0+\frac{24n^2-13n+2}{2(n-2)}+\varepsilon, \quad &\text{if }n\geq 15.
    \end{cases}
\end{equation}
Assume that $C$ is $\vartheta$-good, where $\vartheta$ satisfies %done3
\begin{equation}
    \label{good_demand_1}
    e(n)\frac{2(n-3)}{n+4}< 3\vartheta+\frac{29n^2-38n-2}{(n-2)(n+4)}.
\end{equation}
Then there exists a positive constant $\mathfrak{I}$ satisfying $\mathfrak{I} \gg \rho^{n-1}M^{-1-\frac{2}{n-2}}$ such that
\begin{equation*}
    \begin{split}
        \int_{\mathfrak{M}} S(\alpha) d\alpha=
        &\mathfrak{S}(P_0)\mathfrak{I}P^{n-3}\\
        &+O\left(P^{n-1}M^{3eoP_0-\frac{(n-1)(2n+1)}{n-2}}T\right)\\
        &+O\left(P^{n-\frac{9}{2}}M^{\frac{-2n^2+8n+3}{n-2}+\varepsilon}T^{-\frac{1}{2}}\right),
    \end{split}
\end{equation*}
\end{lemma}
We will prove this Lemma in section 5.

For the contribution from the minor arcs, we have the following Lemma.
\begin{lemma}
    \label{contribution_minor_normal}
Assume that $n \geq 14,$  
\begin{equation}  %done3
    \label{lemma_3.2_right_Q_1}
    e(n) \geq \frac{4(n+4)}{n^2-7n-72}eoP_0+\frac{77n^3+274n^2-746n+16}{2(n-2)(n^2-7n-72)}+\varepsilon,
\end{equation}
\begin{equation}
    \label{Unexpected pregnancy}
    e(n) \geq \frac{77n^4-221n^3+180n^2+70n-96}{2(n-2)(n^3-16n^2+33n-16)}+\varepsilon.
\end{equation}
Assume that $C$ is $\vartheta$-good, where $\vartheta$ satisfies 
%%%%
\begin{equation} %done3
    \label{minor_requir_vartheta_specific_1}
    3e(n)+eoP_0+\varepsilon \leq \frac{\vartheta(n+8)}{4}-\frac{6n^2-14n+1}{n-2} ,
\end{equation}
%%%%
\begin{equation}  %done3
    \label{good_demand_2}
    e(n)\frac{7n}{n+4}< n\vartheta-\frac{n(12n^2-29n-1)}{(n+4)(n-2)}-\varepsilon,
\end{equation}
%%%%
\begin{equation} %done3
    \label{good_demand_3}
    5e(n) < (n-1)\vartheta-\frac{12n^2-24n-1}{n-2}-\varepsilon,
\end{equation}
%%%%
\begin{equation} %done3
    \label{the_critical_when_general}
    e(n)\frac{n^2-6n+4}{n^2-11n+8}+\varepsilon
    \leq \frac{\vartheta n}{4}-\frac{6n^4-14n^3+13n^2+10n-12}{(n-2)(n^2-11n+8)}.
\end{equation}
Then we have
\begin{equation*}
    \int_{\mathfrak{m}}|S(\alpha)|d\alpha \ll P^{n-3-\varepsilon}M^{\frac{-8n^2+12n+1}{n-2}}.
\end{equation*}
\end{lemma}

The upper bound of the integration in Lemma \ref{contribution_minor_normal} is chosen in this way to ensure that the contribution of the minor arcs will be smaller than the main term in the following proof of the two theorems. We will prove these two Lemmas in section 3.

For the truncated singular series, we will prove the following Lemma in section 4.

\begin{lemma}
    \label{series_lower_bound}
Let $\varepsilon > 0$ and assume that $n \geq 14.$ Suppose that \eqref{when_not_good} does not hold and that $C$ is $\vartheta$-good. If $\vartheta = \infty$ then
\begin{equation*}
    \mathfrak{S}(P_0) \gg M^{-6n-\varepsilon}-M^{\frac{n}{6}}P_0^{2-\frac{n}{6}+\varepsilon}.
\end{equation*}
If $\vartheta < \infty$ and $\delta$ satisfies
\begin{equation}
    \label{circle_1}
    2<\delta<\frac{n}{6},
    \quad
    \frac{2n}{n-6\delta}<1+2\vartheta,
\end{equation}
with $P_{0} \ll M^{1+2\vartheta},$ equivalently, 
\begin{equation}
    \label{circle_2}
    eoP_0 \leq 1+2\vartheta,
\end{equation}
then we have
\begin{equation*}
    \mathfrak{S}(P_0) \gg M^{-6n-\varepsilon}-M^{\frac{n\delta}{n-6\delta}}P_0^{2-\delta+\varepsilon}.
\end{equation*}
\end{lemma}

We now have everything in place to prove Theorem \ref{theorem_normal}, in which the cubic form is general. 
To ensure $\mathfrak{S}(P_0) > 0,$ we need
\begin{equation*}
     M^{-6n-\varepsilon} \gg M^{\frac{n\delta}{n-6\delta}}P_0^{2-\delta+\varepsilon},
\end{equation*}
which holds providing
\begin{equation}
    \label{G_normal}
    P_0 > M^{\frac{1}{\delta-2}\left(\frac{n\delta}{n-6\delta}+6n\right)+\varepsilon}.
\end{equation}
We regard the exponent in the right side of \eqref{G_normal} as a function of $\delta \in \left(2, \frac{n}{6}\right)$ and record it as $f_n(\delta),$ which has a derivative with respect to $\delta,$ 
\begin{equation*}
    \frac{\mathrm{d} f_n(\delta)}{\mathrm{d} \delta}=\frac{2n}{(\delta-2)^2(n-6\delta)^2}\bigg(-105\delta^2+36n\delta-n(3n+1)\bigg).
\end{equation*}
It is optimal to take the minimum value at
\begin{equation*}
    \delta=\delta_0(n) =\frac{6n}{35}-\frac{\sqrt{3n(3n-35)}}{105},
\end{equation*}
where
\begin{equation*}
    \delta_0(n) \in \left(2, \frac{n}{6}\right),
\end{equation*}
and
\begin{equation}  
\label{f_n_extremum}
f_n\big(\delta_0(n)\big)=\frac{3675n\sqrt{3n(3n-35)}}{(\sqrt{3n(3n-35)}-18n+210)(3n-6\sqrt{3n(3n-35)})}. 
\end{equation}
So \eqref{G_normal} holds provided that
\begin{equation}
    \label{P_0_another_requirment}
    P_{0} \geq M^{\frac{3675n\sqrt{3n(3n-35)}}{(\sqrt{3n(3n-35)}-18n+210)(3n-6\sqrt{3n(3n-35)})}+\varepsilon},
\end{equation}
which is satisfied by \eqref{P_0_vartheta_definition}.
Comparing $\mathfrak{S}(P_0)\mathfrak{I}P^{n-3}$ with the error terms from the major arcs, it need to ensure  
\begin{equation*}
    P^{n-3-\varepsilon}M^{\frac{-8n^2+12n+1}{n-2}} \gg P^{n-1}M^{3eoP_0-\frac{(n-1)(2n+1)}{n-2}}T,
\end{equation*}
and
\begin{equation*}
    P^{n-3-\varepsilon}M^{\frac{-8n^2+12n+1}{n-2}} \gg P^{n-\frac{9}{2}}M^{\frac{-2n^2+8n+3}{n-2}+\varepsilon}T^{-\frac{1}{2}},
\end{equation*}
which hold provided that
\begin{equation}
\label{T_upper_bound_from_major_integral}
    P^{-2-\varepsilon}M^{-3eoP_0-\frac{n(6n-11)}{n-2}} \gg T,
\end{equation}
and
\begin{equation}
\label{u_need_from_major_integral_error_term}
     T \gg P^{-3+\varepsilon}M^{\frac{4(3n^2-2n+1)}{n-2}}.
\end{equation}
Our choice \eqref{u_normal_definition} of $T$ ensures \eqref{u_need_from_major_integral_error_term}. Furthermore, \eqref{T_upper_bound_from_major_integral} holds provided that
\begin{equation}
\label{for_the_upper_bound_of_T}
    e(n) \geq
    \begin{cases}
    2154.556, \quad &\text{if } n =14,\\
    3eoP_0+\frac{18n^2-19n+4}{n-2}+\varepsilon, \quad &\text{if }n\geq 15.
    \end{cases}
\end{equation}

And considering the requirements of the minor arcs, we need \eqref{lemma_3.2_right_Q_1} and \eqref{Unexpected pregnancy} hold. In addition, we also need to ensure that the conditions related to $\vartheta$ are valid, including 
\eqref{good_demand_1},
\eqref{minor_requir_vartheta_specific_1}, \eqref{good_demand_2}, 
\eqref{good_demand_3}, \eqref{the_critical_when_general}, \eqref{circle_1}, 
\eqref{circle_2}.

The final result of Theorem \ref{theorem_normal} comes from 
\eqref{major_arcs_disjoint_vartheta},
\eqref{The child I want to abandon},
\eqref{lemma_3.2_right_Q_1},
\eqref{Unexpected pregnancy},
\eqref{for_the_upper_bound_of_T},
and
\eqref{good_demand_1},
\eqref{minor_requir_vartheta_specific_1}, \eqref{good_demand_2}, 
\eqref{good_demand_3}, \eqref{the_critical_when_general}, \eqref{circle_1}, 
\eqref{circle_2}.

First, regardless of the restriction of $\vartheta$-good, we temporarily write
\begin{equation*}
    P=M^{E(n)}.
\end{equation*}
$E(n)$ is determined by \eqref{major_arcs_disjoint_vartheta},
\eqref{The child I want to abandon},
\eqref{lemma_3.2_right_Q_1}, \eqref{Unexpected pregnancy},
\eqref{for_the_upper_bound_of_T}, where \eqref{for_the_upper_bound_of_T} is the strongest among them.

For the remaining restrictive conditions that $\vartheta$ should satisfy, we substitute $e(n)=E(n), \delta=\delta_0(n)$ into \eqref{good_demand_1},
\eqref{minor_requir_vartheta_specific_1}, \eqref{good_demand_2}, 
\eqref{good_demand_3}, \eqref{the_critical_when_general}, \eqref{circle_1}, 
\eqref{circle_2}. 
After some calculations, we can know that \eqref{the_critical_when_general} is the strongest among them. Moreover, 
the results of $14 \leq n \leq 18$  can be given in Table \ref{tab2}.

\begin{table}[h!]\footnotesize\tabcolsep 16pt
    \begin{center}
    \caption{The case of $14 \leq n \leq 18$}\label{tab2}\vspace{-2mm}
    \end{center}
    \begin{center}
    \resizebox{\linewidth}{!}{
        \begin{tabular}{l|c|c|c|c|c|c|c|}
    		\hline\hline\noalign{\smallskip}	
    		$n$ & $14$ & $15$ & $16$ & $17$ & $18$  \\
    		\noalign{\smallskip}\hline\noalign{\smallskip}
    		$E(n)$ & $1389.187$ & $1123.474$ & $984.990$ & $904.562$ & $855.213$ \\
      \noalign{\smallskip}\hline\noalign{\smallskip}
    		$\delta=\delta_0(n)$ & $2.237$ & $2.369$ & $2.505$ & $2.642$ & $2.781$  \\

      \noalign{\smallskip}\hline\noalign{\smallskip}
    		$\vartheta$ & $1520.904$ & $835.543$ & $622.688$ & $523.934$ & $462.548$\\

      \noalign{\smallskip}\hline\noalign{\smallskip}
    		$\left\lceil\frac{n^2-1}{2}+\frac{\vartheta n(n-1)}{2}\right\rceil$ & $138500$ & $87844$ & $74851$ & $71400$ & $70932$ \\
    		\noalign{\smallskip}\hline
    	\end{tabular}
     }
    \end{center}
\end{table}

Due to method limitations, the results may get worse when $n>18$. To avoid this, we just need to note that if the cubic form has $n > 18$ variables, we can always set $n-18$ of the variables equal to zero in order to obtain a cubic form in exactly $18$ variables. 

Hence the Theorem \ref{theorem_normal} follows.

\section{The Minor Arcs}

In this section, we consider the contribution from the minor arcs. For the integral over minor arcs, we follow the treatment of Heath-Brown \cite{heath2007cubic}, so sometimes we only give the differences. Throughout this section, we assume that $n \geq 14$.

Let 
\begin{equation}%done 3
    \label{Q}
    Q=\min\left\{P^{\frac{2(n-3)}{n+4}-\varepsilon}M^{\frac{-2(14n^2-20n+3)}{(n+4)(n-2)}}, P^{\frac{3}{2}}M^{\frac{-5(n+1)}{2(n-2)}}\right\}.
\end{equation}
The first restriction in the curly brackets of \eqref{Q} comes from the estimate of $E$, which will be presented in \eqref{E_definition}.
For any $\alpha \in \mathfrak{m},$ by Dirichlet's approximation theorem, there exist coprime integers $0 \leq a \leq q \leq Q$ such that 
\begin{equation}
    \label{Dirichlet_bijin}
    \left| \alpha-\frac{a}{q} \right| \leq \frac{1}{qQ}.
\end{equation}
Since $\alpha \notin \mathfrak{R},$ so either
\begin{equation}
    \label{fenmuda}
    q > P_{0}
\end{equation}
or
\begin{equation}
    \label{banjingda}
    \left| \alpha-\frac{a}{q} \right| > T.
\end{equation}

Suppose that $H=H(R,\phi,\pm)$ is a positive integer with $H \leq P,$ and consider 
\begin{equation*}
    \sum(R,\phi,\pm) := \sum_{R < q \leq 2R}\sum_{\substack{a \leq q\\[3pt](a,q)=1}}\int_{\phi}^{2\phi} \left| S(\frac{a}{q} \pm \nu) \right| d\nu
\end{equation*}
with
\begin{equation}
    \label{Dirichlet_tiaojian}
    \phi \leq (RQ)^{-1}\text{ and }R \leq Q.
\end{equation}

\begin{lemma}
Let $\varepsilon > 0$ and assume that $\rho P \geq 1.$ Assume that $\alpha = \frac{a}{q}+z$ for coprime  integers $0 \leq a 
\leq q$ and that $C$ is $\vartheta$-good. Then we have
\begin{equation*}
    S(\alpha) \ll (\rho P)^{n+\varepsilon}\left(\frac{1}{\rho^2 P^2}+Mq|z|+\frac{q}{\rho^3 P^3}+\frac{1}{q}\min\left\{M,\frac{1}{|z|\rho^3 P^3}\right\}+\frac{1}{M^{2\vartheta}}\right)^{\frac{n}{8}}.
\end{equation*}
\end{lemma}

This is the Lemma 2 of \cite{browning2009least}.

Since $A^{\frac{1}{2}} \leq B +AB^{-1}$ for any $A, B >0$, by the above Lemma we can deduce that
\begin{equation}
    \label{3_1_1}
    S(\alpha) \ll (\rho P)^{n+\varepsilon}\left(Mq|z|+\frac{1}{q|z|\rho^3 P^3}+\frac{1}{M^{2\vartheta}}\right)^{\frac{n}{8}}\text{ for }q \leq M^{\frac{1}{2}}\rho^{\frac{3}{2}}P^{\frac{3}{2}}.
\end{equation}
The second item in the curly brackets of \eqref{Q}  guarantees that \eqref{3_1_1} is valid in the estimation of $\sum(R,\phi,\pm)$. So we have
\begin{equation}
    \label{3_1_2}
    \sum(R,\phi,\pm) \ll R^2 \phi (\rho P)^{n+\varepsilon}\left\{ (MR\phi)^{\frac{n}{8}}+(R\phi\rho^3 P^3)^{-\frac{n}{8}}+M^{-\frac{\vartheta n}{4}}\right\}.
\end{equation}

By \eqref{3_1_2}, we have the following result.

\begin{lemma}
    \label{lemma3_2}
We have
\begin{equation*}
    \sum(R,\phi,\pm) \ll P^{n-3-\varepsilon}M^{\frac{-8n^2+12n+1}{n-2}}
\end{equation*}
providing that $\phi \leq (RQ)^{-1},$ 
\begin{equation*}
    P^{-3+\varepsilon}M^{\frac{54n^2-101n-8}{(n-2)(n-8)}}R^{-\frac{n-16}{n-8}} \ll \phi \ll P^{-\frac{24}{n+8}-\varepsilon}M^{\frac{-49n^2+106n+8}{(n+8)(n-2)}}R^{-\frac{n+16}{n+8}},
\end{equation*}
and
\begin{equation}
 \label{minor_requir_vartheta}
    P^{3+\varepsilon}R^2\phi \ll M^{\frac{\vartheta n}{4}-\frac{6n^2-13n-1}{n-2}}.
\end{equation}
\end{lemma}

Now we consider the case of $R \leq P_0$ and $T \leq \phi \leq (RQ)^{-1}.$ 
After some calculation, we deduce that
\begin{equation*}
    \phi \geq T \geq P^{-3+\varepsilon}M^{\frac{54n^2-101n-8}{(n-2)(n-8)}}R^{-\frac{n-16}{n-8}}
\end{equation*}
holds providing
\begin{equation*}
    T \geq
    \begin{cases}
    P^{-3}M^{\frac{54n^2-101n-8}{(n-2)(n-8)}}P_0^{\frac{16-n}{n-8}+\varepsilon}, \quad &\text{if }14 \leq n \leq 15,\\
        \quad\\
    P^{-3}M^{\frac{54n^2-101n-8}{(n-2)(n-8)}+\varepsilon}, \quad &\text{if }n\geq 16.
    \end{cases}
\end{equation*}
Substitute the definition \eqref{P_0_vartheta_definition} of $P_0,$ specifically, 
\begin{equation}   %done3
\label{u_from_small_denominator_large_distance}
    T \geq
    \begin{cases}
        P^{-3+\varepsilon}M^{320.239}, \quad &\text{if } n =14,\\
        P^{-3+\varepsilon}M^{167.972}, \quad &\text{if } n =15,\\
        P^{-3+\varepsilon}M^{\frac{54n^2-101n-8}{(n-2)(n-8)}+\varepsilon}, \quad &\text{if }n\geq 16.
    \end{cases}
\end{equation}
Our choice  \eqref{u_normal_definition} of $T$ satisfies the above.

And
\begin{equation*}
    \phi \leq (RQ)^{-1} \leq P^{-\frac{24}{n+8}-\varepsilon}M^{\frac{-49n^2+106n+8}{(n+8)(n-2)}}R^{-\frac{n+16}{n+8}}
\end{equation*}
holds providing
$$
P_0^{\frac{8}{n+8}} \leq P^{\frac{-24}{n+8}-\varepsilon}M^{\frac{-49n^2+106n+8}{(n+8)(n-2)}} Q.
$$
Recalling the definition \eqref{P_0_vartheta_definition} of $P_0,$ \eqref{Q} of $Q,$ the above inequality holds provided that 
\begin{equation*}  %done3
    %\label{lemma_3.2_right_Q_1}
    e(n) \geq \frac{4(n+4)}{n^2-7n-72}eoP_0+\frac{77n^3+274n^2-746n+16}{2(n-2)(n^2-7n-72)}+\varepsilon,
\end{equation*}
and 
\begin{equation*}  %done3
    %\label{lemma_3.2_right_Q_2}
    e(n) \geq \frac{16}{3(n-8)}eoP_0+\frac{103n^2-167n+24}{3(n-8)(n-2)}+\varepsilon,
\end{equation*}
the first of which two inequalities is stronger and has been recorded as \eqref{lemma_3.2_right_Q_1}.
Noting that
$$
\phi \ll P^{-\frac{24}{n+8}-\varepsilon}M^{\frac{-49n^2+106n+8}{(n+8)(n-2)}}R^{-\frac{n+16}{n+8}}
$$
in this case,
\eqref{minor_requir_vartheta} turns to be 
\begin{equation*}  %done3
\begin{split}
    3e(n)+eoP_0+\varepsilon \leq \frac{\vartheta(n+8)}{4}-\frac{6n^2-14n+1}{n-2} ,
\end{split}
\end{equation*}
which has been recorded as \eqref{minor_requir_vartheta_specific_1}.
Hence, by Lemma \ref{lemma3_2} we get the following result.

\begin{lemma} %done3
    \label{lemma_3_3}
    Assuming that 
    \eqref{lemma_3.2_right_Q_1}
    and \eqref{minor_requir_vartheta_specific_1}
       hold, then we have
\begin{equation*}
    \sum(R,\phi,\pm) \ll P^{n-3-\varepsilon}M^{\frac{-8n^2+12n+1}{n-2}}
\end{equation*}
providing that $R \leq P_0$ and 
\begin{equation*}
    T \leq \phi \leq (RQ)^{-1}.
\end{equation*}
\end{lemma}

From now on, we assume that $R \leq P_{0}.$ Let $H \leq \rho P$ be a positive integer to be choose later, and $f(\alpha)=e(\alpha C(\boldsymbol{x})).$ Then
\begin{equation*}
    H^nS(\alpha)=\sum_{\boldsymbol{h}}\sum_{\boldsymbol{x}+\boldsymbol{h} \in \mathbb{Z}^{n}\cap P\mathscr{B}} f(\boldsymbol{x}+\boldsymbol{h}),
\end{equation*}
where the sum is for vectors with $1 \leq h_i \leq H$ for each $i$. We rewrite this as
\begin{equation*}
    H^nS(\alpha)=\sum_{\boldsymbol{x}}\sum_{\boldsymbol{x}+\boldsymbol{h} \in \mathbb{Z}^{n}\cap P\mathscr{B}} f(\boldsymbol{x}+\boldsymbol{h}).
\end{equation*}
Since $\mathscr{B}=\mathscr{B}(\boldsymbol{z},\rho) =\prod_{i=1}^{n}[z_i-\rho,z_i+\rho]$ and $H \leq \rho P,$ it follows that $\boldsymbol{x}+\boldsymbol{h} \notin \mathbb{Z}^{n}\cap P\mathscr{B}$ unless $x_i \in [Pz_i-P\rho-H,Pz_i+P\rho-1].$ Hence by Cauchy's inequality we have
\begin{equation}
    \label{from_Cauchy}
    H^{2n}\left|S(\alpha)\right|^2 \leq (2P\rho+H)^n\sum_{\boldsymbol{x}}\left| \sum_{\boldsymbol{x}+\boldsymbol{h} \in \mathbb{Z}^{n}\cap P\mathscr{B}} f(\boldsymbol{x}+\boldsymbol{h}) \right|^2.
\end{equation}
Expanding the square we have
\begin{equation}
    \label{3_3_2}
    H^{2n}\left|S(\alpha)\right|^2 \leq (3P\rho)^n\sum_{\boldsymbol{h}} \omega(\boldsymbol{h})\sum_{\boldsymbol{y},\boldsymbol{y}+\boldsymbol{h} \in \mathbb{Z}^{n}\cap P\mathscr{B}} f(\boldsymbol{y}+\boldsymbol{h})\overline{f(\boldsymbol{y})},
\end{equation}
where the sum over $\boldsymbol{h}$ is for $\left\|\boldsymbol{h}\right\| \leq H,$ and
\begin{equation}
    \label{omega}
    \omega(\boldsymbol{h}):=\#\{\boldsymbol{h}_1, \boldsymbol{h}_2 : \boldsymbol{h}=\boldsymbol{h}_1-\boldsymbol{h}_2\} \leq H^n.
\end{equation}
Hence
\begin{equation}
    \label{3_3_4}
    \left|S(\alpha)\right|^2 \ll H^{-n}(\rho P)^n\sum_{\boldsymbol{h}}|T(\boldsymbol{h},\alpha)|,
\end{equation}
where
\begin{equation}
    \label{T(h,alpha)_definition}
    T(\boldsymbol{h},\alpha):=\sum_{\boldsymbol{y},\boldsymbol{y}+\boldsymbol{h} \in \mathbb{Z}^{n}\cap P\mathscr{B}}f(\boldsymbol{y}+\boldsymbol{h})\overline{f(\boldsymbol{y})}.
\end{equation}

Similar to Lemma 2 of \cite{browning2009least} and Section 3 of \cite{heath2007cubic}, we have 
\begin{equation}
    \label{T(h,alpha)_upper_boubd}
    |T(\boldsymbol{h},\alpha)|^2 \ll (\rho P)^{n+\varepsilon}N(\alpha,P,\boldsymbol{h}),
\end{equation}
where
\begin{equation}
    \label{N(alpha,P,h)_definition}
    N(\alpha,P,\boldsymbol{h}):=\#\{\boldsymbol{w}\in \mathbb{Z}^{n}:|\boldsymbol{w}|<\rho P,\left\|6\alpha\boldsymbol{B}_i(\boldsymbol{h};\boldsymbol{w})\right\| < (\rho P)^{-1} , \forall i \leq n\},
\end{equation}
and
\begin{equation}
    \label{Bi(h,w)_definition}
    \boldsymbol{B}_i(\boldsymbol{h},\boldsymbol{w}):=\sum_{j,k}c_{ijk}x_jx_k.
\end{equation}
By Lemma 2.2 of \cite{heath2007cubic}, we have
\begin{equation}
    \label{N(alpha,P,h)_upper_bound}
    N(\alpha,P,\boldsymbol{h})\ll Z^{-n}\#\{\boldsymbol{w}\in \mathbb{Z}^{n}:|\boldsymbol{w}|<Z\rho P,\left\|6\alpha\boldsymbol{B}_i(\boldsymbol{h};\boldsymbol{w})\right\| < Z(\rho P)^{-1}, \forall i \leq n\}
\end{equation}
for any $0<Z \leq 1.$

We define
\begin{equation}
    \label{psi_definition}
    \psi := |\theta|+\frac{1}{MH\rho^2P^2}.
\end{equation}

If $Z$ satisfies the following
\begin{equation}
    \label{Z_restriction}
    Z \leq 1,\quad Z \ll (MHq\rho P \psi)^{-1}, \quad\text{and}\quad Z \ll q\rho P \psi,
\end{equation}
where the implicit constants are chosen suitable. Then the conditions of Lemma 2.3 of \cite{heath2007cubic} are satisfied. So we have
\begin{equation*}
    \begin{split}
        N(\alpha,P,\boldsymbol{h})
        &\ll Z^{-n}\#\{\boldsymbol{w}\in \mathbb{Z}^{n}:|\boldsymbol{w}|<Z\rho P, \boldsymbol{B}_i(\boldsymbol{h};\boldsymbol{w}) =0,  \forall i \leq n\}\\
        &\ll Z^{-n}(Z\rho P)^{n-r(\boldsymbol{h})},
    \end{split}
\end{equation*}
where $r(\boldsymbol{h})$ is given in section 2. By \eqref{3_3_4} and \eqref{T(h,alpha)_upper_boubd} we have
\begin{equation*}
    \left|S(\alpha)\right|^2 \ll H^{-n}(\rho P)^n\sum_{\boldsymbol{h}}(\rho P)^{\frac{n}{2}+\varepsilon}\{Z^{-n}(Z\rho P)^{n-r(\boldsymbol{h})}\}^{\frac{1}{2}}.
\end{equation*}
Assuming the cubic form $C$ is $\vartheta$-good, by \eqref{good}, we have

%这里第一次涉及 good
%\boldsymbol{\textcolor{green}{This is the first place which involves good-%condition}}

\begin{equation*}
    \left|S(\alpha)\right|^2 \ll (\rho P)^{2n+\varepsilon}\{(\rho P)^{-\frac{n}{2}}Z^{-\frac{n}{2}}+H^{-n}\}.
\end{equation*}

We choose
\begin{equation*}
    Z\asymp \min \{(MHq\rho P \psi)^{-1}, q\rho P \psi\},
\end{equation*}
which automatically yields $Z \leq 1,$
then we have
\begin{equation}
    \label{S(alpha)_van_der_Corput}
    \left|S(\alpha)\right|^2 \ll (\rho P)^{2n+\varepsilon}\{(MHq\psi)^{\frac{n}{2}}+(\rho P)^{-n}(q\psi)^{-\frac{n}{2}}+H^{-n}\}.
\end{equation}
%这里省略了Li的q \geq M的限制
Noting the trivial upper bound of $S(\alpha),$ we can omit the restriction of $H \geq 1$ here.

It will be optimal to take
\begin{equation}
    \label{H_first_time}
    H= 
    \begin{cases}
        \left\lfloor \left(\frac{q}{M}\right)^{\frac{1}{3}} \right\rfloor,\quad &\text{if }|\theta| \leq q^{-\frac{1}{3}}M^{-\frac{2}{3}}\rho^{-2}P^{-2},\\
                \\
        %\max\left\{\left[\frac{1}{M\rho^2P^2|\theta|}\right], \left[\frac{1}{(Mq|\theta|)^{\frac{1}{3}}}\right]\right\}
         \left\lfloor \frac{1}{M\rho^2P^2|\theta|}+\frac{1}{(Mq|\theta|)^{\frac{1}{3}}}\right\rfloor,\quad &\text{if }|\theta| > q^{-\frac{1}{3}}M^{-\frac{2}{3}}\rho^{-2}P^{-2}.\\
    \end{cases}
\end{equation}
Since 
$$
q\leq Q \leq P^{\frac{2(n-3)}{n+4}-\varepsilon}M^{\frac{-2(14n^2-20n+3)}{(n+4)(n-2)}},
$$
to make the $H$ in \eqref{H_first_time} $\vartheta$-good when $|\theta| \leq q^{-\frac{1}{3}}M^{-\frac{2}{3}}\rho^{-2}P^{-2}$, it will suffice provided that 
\begin{equation*}
    e(n)\frac{2(n-3)}{n+4}< 3\vartheta+\frac{29n^2-38n-2}{(n-2)(n+4)},
\end{equation*}
where the inequality has been recorded as \eqref{good_demand_1}.
%这里用e(n)这个符号可能不太恰当
%\begin{equation*}
    %\label{good_demand_1}
%    e(n)\frac{2(n-3)}{n+4}-\varepsilon \leq %3\vartheta+\frac{5n^3-11n^2-140n+22}{(n-2)(n+4)(n-9)}
%\end{equation*}
%and
%\begin{equation*}
    %2e(n) \leq 3\vartheta+\frac{13n+4}{3(n-2)}
%\end{equation*}
%holds.

Substituting equation \eqref{H_first_time} into inequality \eqref{S(alpha)_van_der_Corput}, we deduce that when $|\theta| \leq q^{-\frac{1}{3}}M^{-\frac{2}{3}}P^{-2},$ we have
\begin{equation}
    \label{for_lemma_3_4}
    \left|S(\alpha)\right|^2 \ll (\rho P)^{2n+\varepsilon}\left(\frac{q^{\frac{n}{2}}}{(\rho P)^{n}}+\left(\frac{M}{q}\right)^{\frac{n}{3}}\right).
\end{equation}

We denote
\begin{equation}
    \label{S(a,q)_definition}
    S(a,q)=\sum_{\boldsymbol{r}(\text{mod }q)} e_q(aC(\boldsymbol{r})).
\end{equation}
Substituting $\rho=1, P=q, \theta=0$  into \eqref{for_lemma_3_4}, we get the following Lemma.

\begin{lemma}
    \label{lemma_S(a,q)_upper_bound}
 Let $\varepsilon >0$ and assume that $C$ is $\vartheta$-good, where $\vartheta$ satisfies \eqref{good_demand_1}. Then for $q \geq M,$ we have
    \begin{equation}
    \label{S(a,q)_upper_bound}
    S(a,q) \ll M^{\frac{n}{6}}q^{\frac{5n}{6}+\varepsilon}.
\end{equation}
\end{lemma}

%\begin{remark}
\begin{xrem}
    The estimate \eqref{S(a,q)_upper_bound} is still valid for $q<M,$ but it is worse than the trivial result 
    $$S(a,q) \leq q^n.$$
\end{xrem}

\begin{lemma}
    \label{choice}
    Either
    \begin{equation}
        \Lambda_{n}(C) \ll M^{\frac{1}{n-2}},
    \end{equation}
    or else there exists constants $c,c' > 0$ and a vector $\mathbf{z}=(\xi, \mathbf{y})$ such that $C(\mathbf{z})=0$ and
    \begin{equation}
        M^{-1-\frac{1}{n-2}} \ll |\xi| \ll M^{\frac{1}{n-2}},\quad |\mathbf{y}| \ll M^{\frac{1}{n-2}},
    \end{equation}
    with
    \begin{equation}
        \partial_1=\frac{\partial C}{\partial x_1}(\xi, \boldsymbol{y}) > \frac{c}{M^{1+\frac{4}{n-2}}}
    \end{equation}
    and
    \begin{equation}
        |\partial_2|=\left|\frac{\partial C}{\partial x_2}(\xi, \boldsymbol{y})\right| > \frac{c'}{M^{2+\frac{7}{n-2}}}.
    \end{equation}
\end{lemma}

This is the Lemma 6 of \cite{browning2009least}.

We define
\begin{equation*}
    M(\alpha,H)=\int_{\alpha-\kappa}^{\alpha+\kappa} |S(\beta)|^2 d\beta,
\end{equation*}
where $\kappa \in (0,1)$ is a parameter to be determined. Let the center point $\boldsymbol{z}$ of $\mathscr{B}$ be as in Lemma \ref{choice}.  Consequently, we have
\begin{equation*}
    G :=\left|\frac{\partial C}{\partial Z_1}(\xi, \boldsymbol{y})\right| \gg \frac{1}{M^{1+\frac{4}{n-2}}}.
\end{equation*}
As in Section 4 of \cite{heath2007cubic}, $\boldsymbol{h}$ is restricted by the conditions $1 \leq h_1 \leq \rho P$ and $1 \leq h_2,...,h_n \leq H.$ After some calculations similar to those in \cite{heath2007cubic} , we get the inequality
\begin{equation}
    \label{M(alpha,H)_upper_bound}
    M(\alpha,H)\ll \frac{(\rho P)^{n-1}}{H^{n-1}} \sum_{\boldsymbol{h}} \left|\sum_{\boldsymbol{y}\in \mathbb{Z}^{n}} I(\boldsymbol{h}, \boldsymbol{y}) \right|,
\end{equation}
where
\begin{equation}
    \label{I(h,y)_definition}
    \begin{split}
    I(\boldsymbol{h}, \boldsymbol{y})
    &=\int_{-\infty}^{+\infty} \exp\left(-\frac{(\beta-\alpha)^2}{\kappa^2}\right)e\left(\beta\{ C(\boldsymbol{y}+\boldsymbol{h})-C(\boldsymbol{y})\}\right)d\beta\\
    &=\sqrt{\pi}\kappa \exp\left(-\pi^2\kappa^2\{ C(\boldsymbol{y}+\boldsymbol{h})-C(\boldsymbol{y})\}^2\right)e\left(\alpha\{ C(\boldsymbol{y}+\boldsymbol{h})-C(\boldsymbol{y})\}\right),
    \end{split}    
\end{equation}
and where the sums over $\boldsymbol{h}$ and $\boldsymbol{y}$ are restricted by the condition that $\boldsymbol{y}+\boldsymbol{h}$ and $\boldsymbol{y}$ belong to $P\mathscr{B}.$ By \eqref{rho,z}, we have
\begin{equation*}
    C(\boldsymbol{y}+\boldsymbol{h})-C(\boldsymbol{y}) \gg \frac{|h_1|P^2}{M^{1+\frac{4}{n-2}}}.
\end{equation*}
From now on we choose 
\begin{equation*}
    \kappa \asymp \frac{\mathcal{L}^2M^{1+\frac{4}{n-2}}}{P^2H},
\end{equation*}
so that when $|h_1| \geq H$, we will have
\begin{equation*}
    C(\boldsymbol{y}+\boldsymbol{h})-C(\boldsymbol{y}) \gg \frac{\mathcal{L}^2}{\kappa}.
\end{equation*}
We therefore conclude that the contribution to \eqref{M(alpha,H)_upper_bound} arising from those terms with $|h_1| \geq H$ is $O (1)$.

We may now deduce from \eqref{M(alpha,H)_upper_bound} that
\begin{equation}
    \label{M(alpha,H)_in-process_upper_bound}
    M(\alpha,H)\ll 1+\frac{(\rho P)^{n-1}}{H^{n-1}} \sum_{|h_i| \leq H} \left|\sum_{\boldsymbol{y}\in \mathbb{Z}^{n}} I(\boldsymbol{h}, \boldsymbol{y}) \right|.
\end{equation}
Recalling the integration definition of $I(\boldsymbol{h},\boldsymbol{y})$ in \eqref{I(h,y)_definition}, the range $|\beta-\alpha| \geq \kappa \mathcal{L}$ of 
 $I(\boldsymbol{h},\boldsymbol{y})$ trivially contributes $O(1)$ in total in \eqref{M(alpha,H)_in-process_upper_bound}, whence
\begin{equation*}
    M(\alpha,H)\ll 1+\frac{(\rho P)^{n-1}}{H^{n-1}} \sum_{|h_i| \leq H} \int_{\alpha-\kappa\mathcal{L}}^{\alpha+\kappa\mathcal{L}} \left|T(\boldsymbol{h},\beta) \right| d\beta,
\end{equation*}
where $T(\boldsymbol{h},\beta)$ is given by \eqref{T(h,alpha)_definition}. By \eqref{T(h,alpha)_upper_boubd}, we have
\begin{equation}
    \label{M(alpha,H)_final_upper_bound}
    M(\alpha,H)\ll 1+\frac{\kappa (\rho P)^{3n/2-1+\varepsilon}}{H^{n-1}} \sum_{|h_i| \leq H} \mathop{\max}_{\beta \in \mathcal{I}} N(\beta, P, \boldsymbol{h})^{1/2},
\end{equation}
where
\begin{equation}
    \mathcal{I}=\{ \beta: |\beta-\alpha| \leq \kappa \mathcal{L}\}.
\end{equation}
Next, we consider
\begin{equation}
    \label{A(theta,R,H,P)_definition}
    A(\theta, R, H, P) := \sum_{R <q \leq 2R}\sum_{\substack{a \leq q\\[3pt](a,q)=1}}\sum_{|h_i| \leq H} \mathop{\max}_{\beta \in \mathcal{I}} N(\beta, P, \boldsymbol{h})^{1/2},
\end{equation}
with $\alpha=a/q+\beta.$ For $|h_i| \leq H$ and $|\boldsymbol{w}| < \rho P,$ we have $\boldsymbol{B}_i(\boldsymbol{h};\boldsymbol{w}) \leq cMH\rho P$ for some positive constant $c$. If $\beta \in \mathcal{I}$, and $\left\|6\beta\boldsymbol{B}_i(\boldsymbol{h};\boldsymbol{w})\right\| \leq \frac{1}{\rho P}$, then we have
\begin{equation*}
    \left\|6\alpha\boldsymbol{B}_i(\boldsymbol{h};\boldsymbol{w})\right\| < \frac{1}{\rho P}+\frac{6c\mathcal{L}^3}{PM^{\frac{1}{n-2}}} =\frac{1}{\rho P}\left(1+\frac{6c\mathcal{L}^3}{M^{2+\frac{6}{n-2}}}\right) \ll \frac{1}{\rho P}.
\end{equation*}
Let $\widetilde{P} \asymp \rho P,$ $\widehat{P}=\widetilde{P}/2.$
Then as in section 5 of \cite{heath2007cubic}, we have
\begin{equation*}
    \begin{split}
        \mathop{\max}_{\beta \in \mathcal{I}} N(\beta,& P, \boldsymbol{h})
        \leq \#\left\{\boldsymbol{w}\in \mathbb{Z}^{n}:|\boldsymbol{w}|< \rho P, \left\|6\alpha\boldsymbol{B}_i(\boldsymbol{h};\boldsymbol{w})\right\| <\frac{1}{\widetilde{P}},  \forall i \leq n\right\}\\
        &\ll Z^{-n}\#\left\{\boldsymbol{w}\in \mathbb{Z}^{n}:|\boldsymbol{w}|<Z\widehat{P}, \left\|6\alpha\boldsymbol{B}_i(\boldsymbol{h};\boldsymbol{w})\right\| <\frac{Z}{\widehat{P}},  \forall i \leq n\right\},
    \end{split}
\end{equation*}
for any $0<Z \leq 1.$

We proceed to choose $Z$ in two different ways. When $R < q \leq 2R,$ with $\psi$ as in \eqref{psi_definition}, we first take 
\begin{equation*}
    Z=Z_1\asymp \min\left\{\frac{1}{RMH \rho P \psi}, R \rho P \psi\right\},
\end{equation*}
with suitable implicit constants. Then the condition of Lemma 2.3 in \cite{heath2007cubic} is satisfied, so
\begin{equation*}
    \begin{split}
        \mathop{\max}_{\beta \in \mathcal{I}} N(\beta, P, \boldsymbol{h})
        &\ll Z_1^{-n}\#\{\boldsymbol{w}\in \mathbb{Z}^{n}:|\boldsymbol{w}|<Z_1\widehat{P}, \boldsymbol{B}_i(\boldsymbol{h};\boldsymbol{w})=0,  \forall i \leq n\}\\
        &\ll Z_1^{-n}\{1+(Z_1\widehat{P})^{n-r}\},\\
        &\ll Z_1^{-n}(Z_1\widehat{P})^{n-r},
    \end{split}
\end{equation*}
where $r=r(\boldsymbol{h}).$
Thus
\begin{equation}
    \label{from_Z_1}
    \begin{split}
        \mathop{\max}_{\beta \in \mathcal{I}} N(\beta, P, \boldsymbol{h}) 
        \ll (\rho P)^{n}\left((RMH\psi)^r+\frac{1}{\left((\rho P)^2R\psi\right)^{r}}\right).
    \end{split}
\end{equation}
If we take
\begin{equation*}
    Z=Z_2\asymp \min\left\{1, \frac{1}{RMH\rho P \psi}\right\},
\end{equation*}
with suitable implicit constants. Again by the Lemma 2.3 in \cite{heath2007cubic}, we have
\begin{equation*}
    \mathop{\max}_{\beta \in \mathcal{I}} N(\beta, P, \boldsymbol{h}) \ll Z_2^{-n}\#\{\boldsymbol{w}\in \mathbb{Z}^{n}:|\boldsymbol{w}|<Z_2\widehat{P}, q \mid \boldsymbol{B}_i(\boldsymbol{h}, \boldsymbol{w}), \forall i \leq n\}.
\end{equation*}
Then following the argument in section 5 of \cite{heath2007cubic}, we can get
\begin{equation}
    \label{from_Z_2}
    \mathop{\max}_{\beta \in \mathcal{I}} N(\beta, P, \boldsymbol{h}) \ll (\rho P)^{n}\left(\frac{1}{(\rho P)^{r}}+(RMH\psi)^r+\frac{1}{q_2^r}\right),
\end{equation}
where $q_2$ is defined as in \cite{heath2007cubic}.

We combine \eqref{from_Z_1} with \eqref{from_Z_2} to deduce that
\begin{equation*}
    \begin{split}
        \mathop{\max}&_{\beta \in \mathcal{I}} N(\beta, P, \boldsymbol{h}) \\
        &\ll (\rho P)^{n}\left((RMH\psi)^r+\min\left\{\frac{1}{\left((\rho P)^2R\psi\right)^{r}}, \frac{1}{(\rho P)^{r}}+\frac{1}{q_2^r}\right\}\right)\\
        &\ll (\rho P)^{n}\left((RMH\psi)^r+\frac{1}{(\rho P)^{r}}+\min\left\{\frac{1}{\left((\rho P)^2R\psi\right)^{r}}, \frac{1}{q_2^r}\right\}\right).
    \end{split}
\end{equation*}
Then it follows that
\begin{equation}
    \label{A(theta,R,H,P)_upper_bound}
    \begin{split}
        A(\theta, R, H, P)\ll R(\rho P)^{\frac{n}{2}}\sum_{R<q\leq 2R}\sum_{|h|_i\leq H}
        &\bigg((RMH\psi)^{r/2}+\frac{1}{(\rho P)^{r/2}}\\
        &+\min\left\{\frac{1}{(\rho P)^r(R\psi)^{r/2}}, \frac{1}{q_2^{r/2}}\right\} \bigg).
        \end{split}
\end{equation}
We therefore proceed to consider
\begin{equation*}
    V(\boldsymbol{h}, R, \psi):=\sum_{R<q\leq 2R}\min\left\{\frac{1}{(\rho P)^r(R\psi)^{r/2}}, \frac{1}{q_2^{r/2}}\right\}
\end{equation*}
for a given vector $\boldsymbol{h},$ with $r(\boldsymbol{h})=r.$
Following the argument in section 5 of \cite{heath2007cubic}, we deduce that
\begin{equation*}
    V(\boldsymbol{h}, R, \psi)\ll \mathcal{L}^2(MHR)^{\varepsilon}\frac{R}{(\rho P)^r(R\psi)^{r/2}}\min\left\{1, \left((\rho P)^2\psi\right)^{e(r)}\right\},
\end{equation*}
where
\begin{equation*}
    e(r)=
    \begin{cases}
        0,\quad &\text{if }r=0,\\
        \frac{1}{2},\quad &\text{if }r=1,\\
        1,\quad &\text{if }r\geq 2.
    \end{cases}
\end{equation*}
We now see from \eqref{A(theta,R,H,P)_upper_bound} that
\begin{equation*}
    \begin{split}
        A(\theta, R, H, P)
        &\ll R^2(\rho P)^{\frac{n}{2}}\sum_{|h_i|\leq H}\bigg((RMH\psi)^{r/2}+\frac{1}{(\rho P)^{r/2}}+\frac{V(\boldsymbol{h}, R, \psi)}{R} \bigg)\\
        &\ll R^2(\rho P)^{\frac{n}{2}+\varepsilon}\sum_{|h_i|\leq H}\bigg((RMH\psi)^{r/2}+\frac{1}{(\rho P)^{r/2}}\\
        &\qquad\qquad\qquad\qquad+\frac{1}{(\rho P)^r(R\psi)^{r/2}}\min\left\{1, \left((\rho P)^2\psi\right)^{e(r)}\right\} \bigg).\\
    \end{split}
\end{equation*}
Assuming the cubic form $C$ is $\vartheta$-good, by \eqref{good}, we have

%这里第2次涉及 good
%\boldsymbol{\textcolor{green}{This is the second place which involves good-%condition}}
\begin{equation*}
    \begin{split}
        A(\theta, R, H, P)
        &\ll R^2(\rho P)^{\frac{n}{2}+\varepsilon}\sum_{r=0}^{n}H^r\bigg((RMH\psi)^{r/2}+\frac{1}{(\rho P)^{r/2}}\\
        &\qquad\qquad\qquad\qquad+\frac{1}{(\rho P)^r(R\psi)^{r/2}}\min\left\{1, \left((\rho P)^2\psi\right)^{e(r)}\right\} \bigg)\\
        &\ll R^2(\rho P)^{\frac{n}{2}+\varepsilon}\bigg(1+(RMH^3\psi)^{n/2}+\frac{H^n}{(\rho P)^{n/2}}\\
        &\qquad\qquad\qquad\qquad+\frac{H^n}{(\rho P)^n(R\psi)^{n/2}}\min\left\{1, (\rho P)^2\psi\right\}\\
        &\qquad\qquad\qquad\qquad+\frac{H^2}{(\rho P)^2R\psi}\min\left\{1, (\rho P)^2\psi\right\}\\
        &\qquad\qquad\qquad\qquad+\frac{H}{(\rho P)(R\psi)^{1/2}}\min\left\{1, (\rho P)\psi^{1/2}\right\}\bigg).\\
    \end{split}
\end{equation*}
After a simple comparison, we only need to keep the first, second, forth terms in the brackets. So we conclude the following lemma.
\begin{lemma}
    Under the assumption that \eqref{good} holds, we have
    \begin{equation}
        \begin{split}
            \label{A(theta,R,H,P)_upper_boung_if_good}
            A(\theta, R, H, P)
            \ll R^2(\rho P)^{\frac{n}{2}+\varepsilon}\bigg(&1+(RMH^3\psi)^{n/2}\\
            &+\frac{H^n}{(\rho P)^n(R\psi)^{n/2}}\min\left\{1, (\rho P)^2\psi\right\}\bigg).
        \end{split}
    \end{equation}
\end{lemma}

For $\sum(R,\phi,\pm)$, by Cauchy's inequality we can deduce that
\begin{equation*}
    \sum(R,\phi,\pm)\ll \phi^{\frac{1}{2}}R\left(\sum_{R < q \leq 2R}\sum_{\substack{a \leq q\\[3pt](a,q)=1}}\int_{\phi}^{2\phi} \left| S(\frac{a}{q} \pm \nu) \right|^2 d\nu\right)^{\frac{1}{2}}.
\end{equation*}
We cover the interval $[\phi,2\phi]$ with $ O(1+\frac{\phi}{\kappa})$ intervals of the form $[\theta-\kappa,\theta+\kappa]$ with $\phi\leq\theta\leq2\phi.$ Hence, we have
\begin{equation*}
    \sum(R,\phi,\pm)\ll \phi^{\frac{1}{2}}R\left(1+\frac{\phi}{\kappa}\right)^{1/2}\left(\sum_{R < q \leq 2R}\sum_{\substack{a \leq q\\[3pt](a,q)=1}}M\left(\frac{a}{q}+\theta,H\right)\right)^{\frac{1}{2}},
\end{equation*}
for some $\theta$ in the range $\phi\leq|\theta|\leq2\phi.$
By \eqref{M(alpha,H)_final_upper_bound} and \eqref{A(theta,R,H,P)_definition}, we have
\begin{equation*}
    \sum(R,\phi,\pm)\ll \phi^{\frac{1}{2}}R\left(1+\frac{\phi}{\kappa}\right)^{1/2}\left(R^2+\frac{\kappa (\rho P)^{3n/2-1+\varepsilon}}{H^{n-1}}A(\theta,R,H,P)\right)^{\frac{1}{2}}.
\end{equation*}
Then by \eqref{A(theta,R,H,P)_upper_boung_if_good}, we have
\begin{equation*}
    \sum(R,\phi,\pm)\ll \phi^{\frac{1}{2}}R^2\left(1+\frac{\phi}{\kappa}\right)^{1/2}\left(1+\frac{\kappa }{H^{n-1}}(\rho P)^{2n-1+\varepsilon}E\right)^{\frac{1}{2}},
\end{equation*}
where
\begin{equation}
    \label{E_definition}
    E=1+(RMH^3\psi)^{n/2}+\frac{H^n}{(\rho P)^n(R\psi)^{n/2}}(\rho P)^2\psi,
\end{equation}
$\psi$ is given in \eqref{psi_definition}.
Noting that
\begin{equation*}
    \frac{\psi}{\kappa}=\frac{\phi}{\kappa}+\frac{M^{2+\frac{6}{n-2}}}{\mathcal{L}^2} \gg \frac{\phi}{\kappa}+1,
\end{equation*}
we deduce that
\begin{equation*}
    \begin{split}
        \sum(R,\phi,\pm)
        &\ll \phi^{\frac{1}{2}}R^2\left(\frac{\psi}{\kappa}\right)^{1/2}\left(1+\frac{\kappa }{H^{n-1}}(\rho P)^{2n-1+\varepsilon}E\right)^{\frac{1}{2}}\\
        &\ll \phi^{\frac{1}{2}}R^2\psi^{1/2}\left(\frac{(\rho P)^2HM^{3+\frac{6}{n-2}}}{\mathcal{L}^2}+\frac{(\rho P)^{2n-1+\varepsilon}E}{H^{n-1}}\right)^{\frac{1}{2}},\\
        &\ll \frac{\phi^{\frac{1}{2}}R^2\psi^{1/2}(\rho P)^{n-\frac{1}{2}+\varepsilon}}{H^{\frac{n-1}{2}}}E^{\frac{1}{2}}.
    \end{split}
\end{equation*}

We will have
\begin{equation*}
    \frac{\phi^{\frac{1}{2}}R^2\psi^{1/2}(\rho P)^{n-\frac{1}{2}+\varepsilon}}{H^{\frac{n-1}{2}}}E^{\frac{1}{2}} \ll P^{n-3-\varepsilon}M^{\frac{-8n^2+12n+1}{n-2}},
\end{equation*}
%\boldsymbol{\textcolor{green}{Do I need to explain something about the upper %bound?}}
providing that
\begin{equation}
    \label{if_E<<1}
    \frac{\phi^{\frac{1}{2}}R^2\psi^{1/2}(\rho P)^{n-\frac{1}{2}+\varepsilon}}{H^{\frac{n-1}{2}}} \ll P^{n-3-\varepsilon}M^{\frac{-8n^2+12n+1}{n-2}},
\end{equation}
and
\begin{equation}
    E \ll 1.
\end{equation}
To ensure \eqref{if_E<<1} holds, it is sufficient to choose $H$
satisfying
\begin{equation*}
    P^{5+\varepsilon}R^4\phi^2\rho^{2n-1} \ll H^{n-1}M^{\frac{2(-8n^2+12n+1)}{n-2}},
\end{equation*}
and
\begin{equation*}
    P^{3+\varepsilon}R^4\phi\rho^{2n-3} \ll H^{n}M^{\frac{2(-8n^2+12n+1)}{n-2}+1}.
\end{equation*}
We therefore take
\begin{equation*}
    H=P^{\varepsilon}\max\left\{ 1, (P^5R^4\phi^2M^{\frac{12n^2-24n-1}{n-2}})^{\frac{1}{n-1}}, (P^3R^4\phi M^{\frac{3(4n^2-7n+1)}{n-2}})^{\frac{1}{n}}\right\},
\end{equation*}
More specifically, if we set
\begin{equation*}
    \phi_0 := P^{-\frac{2n+3}{n+1}}M^{\frac{-9n^2+25n-3}{(n-2)(n+1)}}R^{-\frac{4}{n+1}},
\end{equation*}
we take
\begin{equation}
    \label{H_sacond_time}
    H=
    \begin{cases}
        P^{\varepsilon}\max\left\{ 1, (P^3R^4\phi M^{\frac{3(4n^2-7n+1)}{n-2}})^{\frac{1}{n}}\right\},\quad &\text{if }\phi \leq \phi_0,\\
        P^{\varepsilon}\max\left\{ 1, (P^5R^4\phi^2M^{\frac{12n^2-24n-1}{n-2}})^{\frac{1}{n-1}}\right\},\quad &\text{if }\phi > \phi_0.
    \end{cases}
\end{equation}
Recalling \eqref{Dirichlet_tiaojian}, 
%% where the treatment is trivial
to make the $H$ in \eqref{H_sacond_time} $\vartheta$-good, it will suffice providing that
\begin{equation*}
    %\label{good_demand_2}
    e(n)\frac{7n}{n+4}< n\vartheta-\frac{n(12n^2-29n-1)}{(n+4)(n-2)}-\varepsilon,
\end{equation*}
and
\begin{equation*}
    %\label{good_demand_3}
    5e(n) < (n-1)\vartheta-\frac{12n^2-24n-1}{n-2}-\varepsilon,
\end{equation*}
where the two inequalities have been respectively recorded as \eqref{good_demand_2} and \eqref{good_demand_3} before.
Now, if $\phi \leq \phi_0,$ then 
$$
\psi \ll \frac{P^{\varepsilon}}{MHP^2},
$$ 
we calculate that
\begin{equation*}
    \begin{split}
        RMH^3\psi 
        &\ll \frac{RMH^3P^{\varepsilon}}{MH(\rho P)^2}\\
        &\ll \frac{RH^2P^{\varepsilon}}{(\rho P)^2}\\
        &\ll \frac{RP^\varepsilon}{(\rho P)^2}\left(1+\left(P^3R^4\phi M^{\frac{3(4n^2-7n+1)}{n-2}}\right)^{\frac{2}{n}} \right)\\
        &\ll \frac{QP^\varepsilon}{(\rho P)^2}+Q^{\frac{n+4}{n}}P^{\frac{6}{n}-2+\varepsilon}M^{\frac{3(4n^2-7n+1)}{n-2}\cdot\frac{2}{n}}\rho^{-2}\\
        & \ll 1
    \end{split}
\end{equation*}
providing that 
\begin{equation*}
    Q\ll P^{\frac{2(n-3)}{n+4}-\varepsilon}M^{\frac{-2(14n^2-20n+3)}{(n-2)(n+4)}}.
\end{equation*}
Similarly when $\phi \geq \phi_0,$ then 
$$
\psi \ll \phi,
$$ 
we have
\begin{equation*}
    \begin{split}
        RMH^3\psi 
        &\ll RM\phi P^\varepsilon\left(1+\left(P^5R^4\phi^2M^{\frac{12n^2-24n-1}{n-2}}\right)^{\frac{3}{n-1}} \right)\\
        &\ll \frac{MP^\varepsilon}{Q}\left(1+\left(P^5R^4\frac{1}{(RQ)^2}M^{\frac{12n^2-24n-1}{n-2}}\right)^{\frac{3}{n-1}} \right)\\
        & \ll 1
    \end{split}
\end{equation*}
providing that
\begin{equation*}
    Q\gg P^{\frac{15}{n-1}+\varepsilon}M^{\frac{37n^2-75n-1}{(n-1)(n-2)}}.
\end{equation*}
The above two cases are satisfied by \eqref{Q}.

We turn now to the condition
\begin{equation}
    \label{E_final_term}
    \frac{H^n}{(\rho P)^n(R\psi)^{n/2}}(\rho p)^2\psi \ll 1.
\end{equation}
When $\phi \leq \phi_0,$ we have 
$$
\psi \geq \frac{1}{MHP^2},
$$
whence \eqref{E_final_term} holds
providing that
\begin{equation*}
    H\ll \frac{R^{\frac{n}{3n-2}}}{M^{\frac{n-2}{3n-2}}}.
\end{equation*}
Substituting \eqref{H_sacond_time} into the above, this can be satisfied providing that

\begin{equation}
 \label{nothing}
    R \gg M^{\frac{n-2}{n}+\varepsilon},
\end{equation}
and
\begin{equation*}
    \phi \leq \min\left\{\phi_0, \phi_1\right\},
\end{equation*}
where
\begin{equation}
    \label{phi_1_definition}
    \phi_1 := \frac{R^{\frac{n^2-12n+8}{3n-2}}}{P^{3+\varepsilon}M^{\frac{37n^3-91n^2+55n-6}{(3n-2)(n-2)}}}.
\end{equation}
Similarly, when $\phi > \phi_0,$ we have $\psi \geq \phi,$ whence \eqref{E_final_term} holds
providing that
\begin{equation*}
    H\ll (\rho P)^{\frac{n-2}{n}}R^{\frac{1}{2}}\phi^{\frac{1}{2}-\frac{1}{n}}.
\end{equation*}
Substituting \eqref{H_sacond_time} into the above, this can be satisfied providing that
\begin{equation*}
    \phi \geq \max\left\{\phi_0, \phi_2\right\},
\end{equation*}
where 
\begin{equation}
    \label{phi_2_definition}
    \phi_2 := \frac{M^{\frac{2(14n^3-29n^2+2)}{(n^2-7n+2)(n-2)}}}{P^{\frac{2(n^2-8n+2)}{n^2-7n+2}-\varepsilon}R^{\frac{n(n-9)}{n^2-7n+2}}}.
\end{equation}

We now discuss the range $R \geq P_0,$ where \eqref{nothing} always holds.
When 
$$
R \geq P^{\frac{3n-2}{n^2-11n+8}+\varepsilon}M^{\frac{37n^3-81n^2+57n-10}{(n-2)(n^2-11n+8)}},
$$
we have $\phi_2 \leq \phi_0 \leq \phi_1.$ So \eqref{E_final_term} is always true in this case. 
For the remaining range
$$
P_0 \leq R < P^{\frac{3n-2}{n^2-11n+8}+\varepsilon}M^{\frac{37n^3-81n^2+57n-10}{(n-2)(n^2-11n+8)}},
$$
we have 
$$
\phi_1 \leq \phi_0 \leq \phi_2,
$$
therefore \eqref{E_final_term} holds unless 
$$
\phi_1 \leq \phi \leq \phi_2.
$$
%On the one hand, using the lower bound of $R$, $R \geq R_0,$ 

We invoke Lemma \ref{lemma3_2} to deal with this intermediate range.

On the one hand,
\begin{equation*}
    P^{-3+\varepsilon}M^{\frac{54n^2-101n-8}{(n-2)(n-8)}}R^{-\frac{n-16}{n-8}} \leq \phi_1
\end{equation*}
holds providing 
\begin{equation}
    \label{intermediate_range_left}
    R \geq M^{\frac{37n^4-225n^3+ 372n^2-268n+64}{(n-2)(n^3-17n^2+54n-32)}+\varepsilon}.
\end{equation}
Recalling the definition \eqref{P_0_vartheta_definition} of $P_0$, \eqref{intermediate_range_left} is always valid in the range $R \geq P_0.$

On the other hand, 
\begin{equation*}
    \phi_2 \leq P^{-\frac{24}{n+8}-\varepsilon}M^{\frac{-49n^2+106n+8}{(n+8)(n-2)}}R^{-\frac{n+16}{n+8}}
\end{equation*}
holds providing
\begin{equation}
    \label{preliminary_right}
    R \leq P^{\frac{n^3-12n^2+22n-8}{5n^2-19n+16}-\varepsilon} M^{-\frac{77n^4-283n^3+368n^2-152n+16}{2(n-2)(5n^2-19n+16)}}.
\end{equation}
When
\begin{equation}
\label{intermediate_range_right}
    R < P^{\frac{3n-2}{n^2-11n+8}+\varepsilon}M^{\frac{37n^3-81n^2+57n-10}{(n-2)(n^2-11n+8)}},
\end{equation}
\eqref{preliminary_right}
holds provided that
\begin{equation*}
    e(n) \geq \frac{77n^4-221n^3+180n^2+70n-96}{2(n-2)(n^3-16n^2+33n-16)}+\varepsilon,
\end{equation*}
which has been recorded as \eqref{Unexpected pregnancy}.
Substituting \eqref{intermediate_range_right}, and using $\phi \leq \phi_2,$
\eqref{minor_requir_vartheta} 
holds providing that
\begin{equation*}
    e(n)\frac{n^2-6n+4}{n^2-11n+8}+\varepsilon
    \leq \frac{\vartheta n}{4}-\frac{6n^4-14n^3+13n^2+10n-12}{(n-2)(n^2-11n+8)},
\end{equation*}
which has been recorded as \eqref{the_critical_when_general}.

So for the range $R \geq P_0,$ we have
\begin{equation*}
    \sum(R,\phi,\pm) \ll P^{n-3-\varepsilon}M^{\frac{-8n^2+12n+1}{n-2}},
\end{equation*}
providing that 
\eqref{Unexpected pregnancy}
\eqref{good_demand_2}, \eqref{good_demand_3} and \eqref{the_critical_when_general}
hold.

Combined with the estimate for the range
$$
R \leq P_0, \quad T \leq \phi \leq (RQ)^{-1},
$$
which has been deduced in Lemma \ref{lemma_3_3}, we are now able to obtain the following Lemma.

\begin{lemma}
    \label{lemma_3_7}
    Suppose that $n\geq 14$, $C$ is $\vartheta$-good, and \eqref{lemma_3.2_right_Q_1}, \eqref{Unexpected pregnancy} hold.
    Then for $\phi \leq \frac{1}{RQ}$ and $R \leq Q,$ we have
    \begin{equation*}
        \sum(R,\phi,\pm) \ll P^{n-3-\varepsilon}M^{\frac{-8n^2+12n+1}{n-2}},
    \end{equation*}
    unless
    \begin{equation}
        R\leq P_0 \quad\text{and}\quad\phi \leq T,
    \end{equation}
    where $\vartheta$ needs to satisfy \eqref{minor_requir_vartheta_specific_1}, \eqref{good_demand_2}, \eqref{good_demand_3}, \eqref{the_critical_when_general}. 
\end{lemma}

We proceed to cover the bulk of the minor arcs by sets of the form
\begin{equation*}
    I(a,q)=\left\{\alpha=\frac{a}{q}+\nu: \nu_0 \leq |\nu| \leq \frac{1}{qQ} \right\},
\end{equation*}
where 
\begin{equation*}
    \nu_0=
    \begin{cases}
        T,\quad &\text{if }q \leq P_0,\\
        \frac{1}{P^n},\quad &\text{otherwise. }
    \end{cases}
\end{equation*}
For the interval $[\frac{a}{q}-\frac{1}{P^n}, \frac{a}{q}+\frac{1}{P^n}],$ by the trivial estimation, we have
\begin{equation*}
    \sum_{q \leq Q}\sum_{\substack{a \leq q\\[3pt](a,q)=1}}\int_{\frac{a}{q}-\frac{1}{P^n}}^{\frac{a}{q}+\frac{1}{P^n}}|S(\alpha)| d\alpha \leq Q^2,
\end{equation*}
which is $o(P^{n-3-\varepsilon}M^{\frac{-8n^2+12n+1}{n-2}}).$
By dyadic method we also have
\begin{equation*}
    \sum_{q \leq Q}\sum_{\substack{a \leq q\\[3pt](a,q)=1}}\int_{I(a,q)}|S(\alpha)| d\alpha \leq \mathcal{L}^2\sum(R,\phi,\pm),
\end{equation*}
for some $R\leq Q,$ some $\phi \leq (RQ)^{-1},$ and some choice of $\pm.$ Hence the Lemma \ref{contribution_minor_normal} follows from the Lemma \ref{lemma_3_7}.

\section{The Truncated Singular Series}

In this section we give the proof of Lemma \ref{series_lower_bound}. Let $\Delta(C)$ be the greatest common divisor of all the $n \times n$ subdeterminants of the $n \times \frac{1}{2}n(n+1)$ matrix formed from the coefficients of $C.$ This is introduced by Davenport \cite{davenport2005analytic} and denoted by $h(C)$ there. It is invariant under any unimodular change of variables, and $\Delta(C)=0$ if and only if $C$ is degenerate. In \cite{lloyd1975bounds}, Lloyd shows that $\Lambda_n(C) \ll M^{n-1}$ if $C$ is degenerate. So we may assume that $\Delta(C)\neq 0.$ By definition one has 
\begin{equation*}
    0 < \Delta(C) \ll M^n.
\end{equation*}

For positive integer $k$, we let
\begin{equation*}
    \rho(p^k):= \# \{ \boldsymbol {x} \in (\mathbb{Z}/p^k\mathbb{Z})^{n} : C(\boldsymbol {x})\equiv 0\ (\text{mod} \ p^k)\},
\end{equation*}
and $\rho^{*}(p^k)$ for the number of the set of non-singular solutions modular $p^k.$

Recalling the definition of $\mathfrak{S}(R),$ we have
\begin{equation*}
    \mathfrak{S}(P_0)=\sum_{q \leq P_0} A(q),
\end{equation*}
in which
\begin{equation}
    \label{A(q)_definition}
    A(q) =\sum_{\substack{1 \leq a \leq q\\[3pt](a,q)=1}} \frac{S(a,q)}{q^n},
\end{equation}
and $S(a,q)$ is given in \eqref{S(a,q)_definition}. It is well-known that
\begin{equation*}
    \sum_{i=0}^{k}A(p^k)= \frac{\rho(p^k)}{p^{k(n-1)}}.
\end{equation*}
We now define the truncated Euler product
\begin{equation*}
    S(P_0) :=\prod_{p \leq P_0}\sum_{i=0}^{k(p)}A(p^i),
\end{equation*}
where $k(p)$ is given by (7.4) of \cite{browning2009least}. The following Lemma gives a uniform lower bound for this quantity.
\begin{lemma}
    Let $\varepsilon>0.$ Then we have
    \begin{equation*}
        S(P_0) \gg M^{-6n-\varepsilon}.
    \end{equation*}
\end{lemma}

This is the Lemma 16 of \cite{bernert2023}.

Define
\begin{equation*}
    \mathscr{L}(P_0):=\left\{q \in \mathbb{N}: q>P_0,\ p^i \mid q \Rightarrow p \leq P_0 \text{ and }i \leq k(p) \right\}.
\end{equation*}
Then we have
\begin{equation*}
    R(P_0) := |\mathfrak{S}(P_0)-S(P_0)| \leq \sum_{q \in \mathscr{L}(P_0)} |A(q)|,
\end{equation*}
where $A(q)$ is given by \eqref{A(q)_definition}. For the upper bound of $R(P_0),$ we have the following results.

\begin{lemma}
\label{lemma_4_2}
    Assume that $C$ is $\infty$-good. Then for $n \geq 14,$ we have
    \begin{equation*}
        R(P_0) \ll M^{\frac{n}{6}}P_0^{2-\frac{n}{6}+\varepsilon}.
    \end{equation*}
\end{lemma}

\begin{proof}
It follows from Lemma \ref{lemma_S(a,q)_upper_bound} that $A(q) \ll M^{\frac{n}{6}}q^{1-\frac{n}{6}+\varepsilon}.$ So we have
\begin{equation*}
    R(P_0) \ll \sum_{q> P_0}|A(q)| \ll \sum_{q> P_0}M^{\frac{n}{6}}q^{1-\frac{n}{6}+\varepsilon} \ll M^{\frac{n}{6}}P_0^{2-\frac{n}{6}+\varepsilon},
\end{equation*}
for $n\geq 14.$
\end{proof}

\begin{lemma}
    Assume that $C$ is $\vartheta$-good, for $\vartheta < \infty.$ Assume furthermore that $P_0 \ll M^{1+2\vartheta}$ and $\delta$ satisfies
    \begin{equation}
        \label{delta_condition}
        2<\delta<\frac{n}{6}, \quad \frac{2n}{n-6\delta}<1+2\vartheta.
    \end{equation}
    Then we have
    \begin{equation*}
        R(P_0) \ll M^{\frac{n\delta}{n-6\delta}}P_0^{2-\delta+\varepsilon}.
    \end{equation*}
\end{lemma}

\begin{proof}
We follow the approach of Lemma 13 in \cite{browning2009least}. Let $q$ be an integer in the interval
\begin{equation}
\label{q_i_range}
    M^{\frac{n}{n-6\delta}+\varepsilon} \leq q \leq M^{1+2\vartheta}.
\end{equation}
In the proof of Lemma \ref{lemma_4_2}, we get $A(q) \ll M^{\frac{n}{6}}q^{1-\frac{n}{6}+\varepsilon}.$ If $q$ satisfies \eqref{q_i_range}, then we have 
\begin{equation*}
    A(q)=O(q^{1-\delta}),
\end{equation*}
uniformly in $M.$

Let
\begin{equation*}
    A= \frac{n}{n-6\delta}+\varepsilon,\quad B=1+2\vartheta.
\end{equation*}
One has $2A <B.$ And for each $q \in \mathscr{L}(P_0),$ Browning and et al \cite{browning2009least} prove that there is a factorisation
\begin{equation*}
    q=q_1...q_tq_{t+1},
\end{equation*}
with $(q_1, q_j)=1$ for each $1 \leq i < j \leq t+1,$ where $q_i$ satisfies \eqref{q_i_range} for $1 \leq i \leq t$ and $q_{t+1}< M^{A}.$ Write $q=q_0q_{t+1},$ with $q_0=q_1...q_t,$ then one has
\begin{equation*}
    R(P_0) \ll M^A \sum_{q_{t+1}<M^A}\sum_{q_0>\frac{P_0}{q_{t+1}}}q_0^{1-\delta+\varepsilon} \ll M^{A\delta}P_0^{2-\delta+\varepsilon}.
\end{equation*}
Then the Lemma follows.

In this Lemma we correct an error in the treatment of Lemma 13 in \cite{browning2009least}, where the exponent of $M$ should be $\frac{n\delta}{n-8\delta}+\varepsilon$ rather than $\frac{n}{n-8\delta}+\varepsilon.$
\end{proof}

\section{The Major Arcs}

In this section we will give the proof of Lemma \ref{contribution_major}. For $\alpha \in \mathfrak{M},$ there exist coprime integers $a, q$ such that $1 \leq a < q \leq P_0,$ and
\begin{equation*}
    \alpha=\frac{a}{q}+ z, \quad |z| \leq T.
\end{equation*}
Let $D$ be the smallest real number exceeding 1 such that $\mathscr{B} \in [-D, D]^n.$ Then we have
\begin{equation}
    1 \leq D \ll 1+|\boldsymbol{z}|,
\end{equation}
and
\begin{equation}
    S(\alpha)=\sum_{\boldsymbol{r}(\text{mod }q)} e_q(aC(\boldsymbol{r}))\sum_{\substack{\boldsymbol{y} \in \mathbb{Z}^n\\[3pt]\boldsymbol{r}
    +q\boldsymbol{y} \in P\mathscr{B}}} e(zC(\boldsymbol{r}
    +q\boldsymbol{y})).
\end{equation}

\begin{lemma}
    \label{lemma_5_1}
    Let $\alpha=a/q+ z \in \mathfrak{M}(a,q),$ for coprime integers $a,q$ satisfying the inequality $0 \leq a<q \leq P_0,$ where $P_0$ is given by \eqref{P_0_vartheta_definition}. Assume that  
    \begin{equation}
    \label{Not very important}
        P_0TM^{\frac{3}{2}}P^2 \ll 1.
    \end{equation}
    Then either $\Lambda_n(C)=O(1),$ or else we have
    \begin{equation*}
        S(\alpha)=\frac{S(a,q)I_P(z)}{q^n}+O\left(q(\rho P)^{n-1}\right),
    \end{equation*}
    where 
    \begin{equation*}
        S(a,q)=\sum_{\boldsymbol{r}(\text{mod }q)} e_q(aC(\boldsymbol{r})),
    \end{equation*}
    and
    \begin{equation*}
        I_P(z)=\int_{P\mathscr{B}} e(zC(\boldsymbol{x})) d\boldsymbol{x}.
    \end{equation*}
\end{lemma}

\begin{proof}
Same as Lemma 5 of \cite{browning2009least}, it is sufficient to note that by \eqref{Not very important}, we have
    \begin{equation*}
        \begin{split}
            q|z|M(DP)^2 \ll P_0 T M P^2 ,
        \end{split}
    \end{equation*}
    which is $O(M^{-\eta})$ for a certain value of $\eta >0$ when $n \geq 14.$
    Substituting the definition \eqref{P_0_vartheta_definition} of $P_0$ and \eqref{u_normal_definition} of $T$, \eqref{Not very important} is equivalent with
    $$
    e(n) \geq
    \begin{cases}
    904.479, \quad &\text{if } n =14,\\
    eoP_0+\frac{24n^2-13n+2}{2(n-2)}+\varepsilon, \quad &\text{if }n\geq 15,
    \end{cases}
    $$
    which has been recorded as \eqref{The child I want to abandon}.
\end{proof}

Following the approach of section 5 in \cite{browning2009least}, with Lemma \ref{lemma_5_1} we deduce that
\begin{equation*}
        \int_{\mathfrak{M}} S(\alpha) d\alpha=
        \mathfrak{S}(P_0)\mathfrak{I}(TP^3)P^{n-3}
        +O\left((\rho P)^{n-1}P_0^{3}T\right),
\end{equation*}
then
\begin{equation}
    \begin{split}
        \int_{\mathfrak{M}} S(\alpha) d\alpha=
        &\mathfrak{S}(P_0)V(0)P^{n-3}\\
        &+O\left((\rho P)^{n-1}P_0^{3}T\right)\\
        &+O\left(|\mathfrak{S}(P_0)|\rho ^{n-2}M^{3+\frac{11}{n-2}}(TP^3)^{-\frac{1}{2}}P^{n-3}\right),
    \end{split}
\end{equation}
where $V(0) \gg \rho^{n-1}M^{-1-\frac{2}{n-2}}.$

For $n \geq 14,$ we assume that $C$ is $\vartheta$-good, where $\vartheta$ satisfies \eqref{good_demand_1}. Then invoking the trivial upper bound of $S(a,q)$ and Lemma \ref{lemma_S(a,q)_upper_bound} respectively, we have
\begin{equation*}
    \begin{split}
        \mathfrak{S}(P_0) 
        &\ll \sum_{q \leq P_0} q^{1-n}|S(a,q)|\\ 
        &\ll \sum_{1 \leq q \leq M} q^{1-n}|S(a,q)|+\sum_{q >M} q^{1-n}|S(a,q)|\\
        &\ll \sum_{1 \leq q \leq M} q+\sum_{q >M}M^{\frac{n}{6}}q^{1-\frac{n}{6}+\varepsilon}\\
        &\ll M^{2+\varepsilon}.
    \end{split}
\end{equation*}
Taking $\mathfrak{I}=V(0)$, then the Lemma \ref{contribution_major} follows.

\subsection*{Acknowledgements}
The authors would like to thank Professor T.D.Browning for his very helpful suggestions and comments.
%We also would like to thank the anonymous referees for his or her comments and suggestions.
%We would like to thank Prof.T.D.Browning for catching an error in an earlier version of the paper.
This work was supported by National Natural Science Foundation of China (Grant No.12171311).

%%%%%%%%%%% To ease editing, use normal size for the references:

\end{document}